\documentclass[a4paper,reqno,12pt]{article}
\usepackage{shuffle}

\usepackage{hyperref}
\usepackage{amsmath,amsfonts,amssymb}
\usepackage{fancyhdr}
\usepackage{mathrsfs,amssymb} % Zusatzzeichen
\usepackage{amsmath}
\usepackage{amssymb}
\usepackage{empheq}
\usepackage{enumerate}
\usepackage{tikz}
\usepackage{float}
\usepackage[a4paper]{geometry}
\usepackage[latin1]{inputenc}
\usepackage[T1]{fontenc}
\usepackage{amsthm}
\usepackage{dsfont}
\usepackage[nottoc]{tocbibind}
\usepackage{color}
\usepackage{fancyhdr,blindtext}
\usepackage[arrow, matrix, curve]{xy}
\usepackage{array}
\usepackage[vcentermath]{youngtab}
\newcolumntype{x}[1]{>{\centering\arraybackslash\hspace{0pt}}m{#1}}

\fancypagestyle{plain}{%
\fancyhf{}}

\renewcommand{\sectionmark}[1]%
{\markboth{#1}{}}

\makeindex

\theoremstyle{definition}
\newtheorem{ex}{\bfseries \upshape Example}[section]

\newtheorem{rem}[ex]{\bfseries \upshape Remark}
\newtheorem{conj}[ex]{\bfseries \upshape Conjecture}
\newtheorem{prop}[ex]{\bfseries \upshape Proposition}
\newtheorem{lem}[ex]{\bfseries \upshape Lemma}
\newtheorem{thm}[ex]{\bfseries \upshape Theorem}
\theoremstyle{plain}

\newtheorem{cor}[ex]{\bfseries \upshape Corollary}

\newenvironment{prf}{\begin{proof}[{\bf Proof}]}{\end{proof}}

\newcommand{\N}{\ensuremath{\mathds{N}}}	
	
\newcommand{\Q}{\ensuremath{\mathds{Q}}}

\newcommand{\dif}{\operatorname{d}}

\newcommand{\grw}{ \operatorname{gr}^{\operatorname{W}}}

\newcommand{\grwl}{ \operatorname{gr}^{\operatorname{W},\operatorname{L}}}

\newcommand{\filw}{ \operatorname{Fil}^{\operatorname{W}}}

\newcommand{\filwle}{ \operatorname{Fil}^{\operatorname{W},\operatorname{L}}}

\DeclareMathOperator{\MD}{\mathcal{MD}}

\DeclareMathOperator{\qMZV}{\mathsf{qMZV}}
\DeclareMathOperator{\QZ}{\mathsf{Z}}

\DeclareMathOperator{\OZ}{\mathsf{Z}}

\numberwithin{equation}{section}

\begin{document}
\title{{ \bf A short note on a conjecture of Okounkov about a $q$-analogue of multiple zeta values}}
\author{{\sc Henrik Bachmann, Ulf K\"uhn}}
\date{\today}
\maketitle

\begin{abstract}
In \cite{ao} Okounkov studies a specific $q$-analogue of multiple zeta values and makes some conjectures on their algebraic structure. In this note we compare Okounkovs $q$-analogues to the generating function for multiple divisor sums defined in \cite{bk}. We also state a conjecture
on their dimensions that complements Okounkovs conjectural formula and present some numerical  
evidences for it.
\end{abstract}
%\tableofcontents

\section{Introduction}
Multiple zeta values are natural generalizations of the Riemann zeta values that are defined for integers $s_1 > 1$ and $s_i \geq 1$ for $i>1$ by
\[ \zeta( s_1 , \dots , s_l ) :=  \sum_{n_1 >n_2 >\dots>n_l>0 } \frac{1}{n_1^{s_1} \dots n_l^{s_l} } \,. \]
Because of its occurence in various fields of mathematics and physics these real numbers 
are of particular interest.
%
%The $\Q$-vector space of all multiple zeta values of weight $k$ is then given by 
%%%\[ \MZ_k := \bigoplus_{\substack{l>0 \\ s_1 + \dots + s_l = k}} <\zeta(s_1,\dots,s_l)>_\Q  \]
%\[ \MZ_k :=\big <\,\zeta(s_1,\dots,s_l) \, \big| \, s_1 + \dots + s_l = k 
%\textrm{ and } l>0 \big>_\Q.  \]
In  \cite{ao} Okounkov  discusses a conjectural connection from enumerative geometry of some Hilbert schemes to
a specific $q$-analogue $\mathsf{Z}(s_1,...,s_l)$ of the multiple zeta-values. He denotes by $\qMZV$ the $\Q$-algebra generated by these. 
In this short note we want to discuss the connection of these  multiple $q$-zeta values to the algebra  $\MD$ of generating functions for multiple divisor sums  $[s_1,..,s_l]$ defined by the authors in \cite{bk}. More precisely we have

\begin{thm} 
Let $\MD^\sharp = \langle\,[s_1,...,s_l] \in \MD\,| s_i>1 \, \forall\,i \mbox{ or } s_1= \emptyset\, \rangle_\Q.$
\begin{enumerate}[i)]
\item   The sub vector space $\MD^\sharp$ is in fact a sub algebra of $\MD$.
\item We have $\qMZV = \MD^\sharp$, in particular the $\Q$-vector space generated by the 
$ \OZ(s_1,...,s_l)$ is closed under multiplication.
\item  We have $ q \frac{d}{d q} \mathsf{Z}(k) \in \qMZV$ for all $k \ge 2$.
\end{enumerate}
\end{thm}

The first two statements are merely a reformulation of results implictly contained in \cite{bk}. The third is 
direct consequence of some explicit formula given in \cite{bk}. It gives some evidence to the
conjecture of Okunkov, that the operator $d$ is a derivation on $\qMZV$.

Acknowledgements: We thank J.~Zhao for pointing out some minor mistakes in the first version of this notes. 

\section{$q$-analogues of multiple zeta values}
In the following we fix a subset  $S \subset \N$, which we consider as the support for
  index entries, i.e. we assume $s_1,\dots,s_l \in S$. 
For each $s\in S$ we let  $Q_s(t) \in \Q[t]$ be a polynomial with $Q_s(0)=0$ and $Q_s(1) \neq 0$.
We set $Q = \left\{ Q_s(t) \right\}_{s\in S}$.
A sum of the form 
\begin{equation} \label{eq:zq}
 Z_Q(s_1,\dots,s_l) := \sum_{n_1 > \dots > n_l > 0} \prod_{j=1}^l \frac{Q_{s_j}(q^{n_j})}{(1-q^{n_j})^{s_j}} 
\end{equation}
with polynomials $Q_s$ as before, defines
 a \emph{$q$-analogue of a multiple zeta-value} of  weight $k=s_1+ \dots +s_l$ and length $l$.  
Observe only because of $Q_{s_1}(0)=0$ this defines an element of $\Q[[q]]$.   
This notion is  due to the  identity
 \begin{align*}
 \lim\limits_{q \rightarrow 1}{(1-q)^k Z_Q(s_1,\dots,s_l)} &= \sum_{n_1 > \dots > n_l > 0} 
  \prod_{j=1}^l     \lim\limits_{q \rightarrow 1}{ 
  \left( Q_{s_j}(q^{n_j}) \frac{(1-q)^{s_j}}{(1-q^{n_j})^{s_j}} \right)}\\
 &=     
  Q_{s_1}(1) \dots Q_{s_l}(1) \cdot \zeta(s_1,\dots,s_l) \,.
  \end{align*}
Here we used that $ \lim\limits_{q \rightarrow 1}{ 
   (1-q)^{s}/(1-q^{n})^{s} = 1/n^s }$ and
with the same arguments as in \cite{bk} Proposition 6.4, the above identity can be justified for all 
$(s_1,...,s_l)$ with
$s_1 > 1$. Related definition for $q$-analogues of multiple zeta values are given in \cite{db}, \cite{yt}, \cite{wz} and \cite{ooz}. 
 It is convenient to define $Z_Q(\emptyset)=1$ and then we denote the vector space spanned by all these elements by
\begin{equation} \label{eq:zqs}
 Z(Q,S) := \big <\,  Z_Q(s_1,\dots,s_l) \big| \,l \ge 0 \mbox{ and } s_1,\dots,s_l \in S \big>_\Q \,. 
\end{equation}
Note by the above convention we have that $\Q$ is contained in this space.

\begin{lem} \label{lem:algebra}
If for each $r,s \in S$ there exists numbers $\lambda_j(r,s) \in \Q$ 
   such that
\begin{equation}\label{eq:reduction} Q_r(t) \cdot Q_s(t) = \sum_{\substack{j \in S \\1 \le j \le r+s}} \lambda_j(r,s) (1-t)^{r+s-j} Q_j(t) \,, \end{equation}
then the vector space $Z(Q,S)$ is a $\Q$-algebra, 
\end{lem}
\begin{prf}
We have to show that $Z_Q(s_1,\dots,s_l) \cdot Z_Q(r_1,\dots,r_m) \in Z(Q,S)$ and illustrate this in the $l=m=1$ case because the higher length case will be clear after this. Suppose there is a representation of the form \eqref{eq:reduction} then it is
\begin{align*}
Z_Q(r) \cdot Z_Q(s) &= \sum_{n_1>0} \frac{Q_{r}(q^{n_1})}{(1-q^{n_1})^{r}} \cdot \sum_{n_2>0} \frac{Q_{s}(q^{n_2})}{(1-q^{n_2})^{s}} \\
&= \sum_{n_1>n_2>0} \dots + \sum_{n_2 >n_1>0} \dots + \sum_{n_1=n_2=n>0} \frac{Q_{r}(q^{n})Q_{s}(q^{n})}{(1-q^{n})^{r+s}} \\
&= Z_Q(r,s)+Z_Q(s,r) + \sum_{j \in S'} \lambda_j  Z_Q(j) \,\in Z(S,Q)\,.
\end{align*}
\end{prf}

We give three examples of $q$-analogues of multiple zeta values, which are currently considered by different authors where just the second and the third will be of interest in the rest of this note. 

\begin{enumerate}[i)]

\item[0)] The polynomials $Q_s^T(t) = t^{s-1}$ are considered in \cite{yt} and sums of the form \eqref{eq:zq} with $s_1>1$ and $s_2, \dots ,s_l \geq 1$ are studied there.   
\item In \cite{bk} the authors choose  $Q^E_s(t) = \frac{1}{(s-1)!} t P_{s-1}(t)$, where the $P_s(t)$ are the eulerian polynomials defined by
\[ \frac{t P_{s-1}(t)}{(1-t)^{s}} = \sum_{d=1}^\infty d^{s-1} t^d  \]
for $s\geq 0$. With this define for all $s_1,\dots,s_l \in \N$
\[ [s_1,...,s_l] := 
\sum_{n_1 > ... > n_l > 0} \prod_{j=1}^l \frac{Q_{s_j}^E(q^{n_j})}{(1-q^{n_j})^{s_j}} \,. \]  
%\sum_{n_1 > ... > n_l > 0} \prod_{j=1}^l \frac{q^{n_j} P_{s_j-1}(q^{n_j})}{(s_j-1)! (1-q^{n_j})^{s_j}} \,. \]  
and set
\[ \MD = Z( \{Q^E_s(t))\}_{s},\N) \,.\]
These \emph{brackets} are generating functions for multiple divisor sums and they occur in the 
Fourier expansion of multiple Eisenstein series.

\item Okounkov chooses the following polynomials in \cite{ao}
\[  Q^O_s(t) = \begin{cases} t^{\frac{s}{2}} & s = 2,4,6,\dots  \\ 
t^{\frac{s-1}{2}} (1+t) & s=3,5,7,\dots . \end{cases} \]
and defines for $s_1,\dots,s_l \in S=\N_{>1}$
\[ \OZ(s) =  \sum_{n_1 > \dots > n_l > 0} \prod_{j=^0}^l \frac{Q^O_{s_j}(q^{n_j})}{(1-q^{n_j})^{s_j}}\,.\] 
We write for the space of the Okounkov $q$-multiple zetas
\[ \qMZV = Z(\{ Q^O_s(t) \}_s ,\N_{>1}) \,. \] 

\end{enumerate}

\begin{prop} \label{prop:quasi-shuffle} For the polynomials above we have
\begin{enumerate}[i)]
\item for $r,s \in \N$ and $Q_j^E(t) = \frac{1}{(s-1)!} t P_{s-1}(t)$
\[  Q_r^E(t) \cdot Q_s^E(t) = \sum_{j=1}^r \lambda^j_{r,s} (1-t)^{r+s-j} Q_{j}^E(t) + \sum_{j=1}^s \lambda^j_{s,r} (1-t)^{r+s-j}  Q_{j}^E(t)  + Q_{r+s}^E(t) \,, \]
where the coefficient $\lambda^j_{a,b}  \in \Q$ for $1 \leq j \leq a$ is given by
\[ \lambda^j_{a,b} = (-1)^{b-1} \binom{a+b-j-1}{a-j} \frac{B_{a+b-j}}{(a+b-j)!} \,. \]
\item for $r,s \in \N_{>1}$  it is
\[ Q^O_r(t) \cdot Q^O_s(t) =  \begin{cases} Q^O_{r+s}(X) &\,, r,s \text{ even or } r+s \text{ odd}  \\ 
4 Q^O_{r+s}(t) + (1-t)^2 Q^O_{r+s-2}(t) &\,, r,s \text{ odd}\,.\end{cases} \]
\end{enumerate}
In particular, because of Lemma \ref{lem:algebra}, the vector spaces     $\MD$  and $\qMZV$  are $\Q$-algebras.
\end{prop}
\begin{prf} In \cite{bk} the claim i) is proven. The cases in ii) are checked easily.  

\end{prf}
\begin{cor}
$\MD^\sharp = Z(\left\{Q^E_s \right\}_{s},\N_{>1})$ is a sub algebra of $\MD$.
\end{cor}
\begin{proof} Using Proposition \ref{prop:quasi-shuffle} it is easy to see that
it suffices to show that      
\[ \lambda^1_{a,b}+\lambda^1_{b,a} = \left( (-1)^{a-1} + (-1)^{b-1} \right) \binom{a+b-2}{a-1} \frac{B_{a+b-1}}{(a+b-1)!} \, \]
vanishes for $a,b > 1$.
This term clearly vanishes when $a$ and $b$ have different parity. In the other case $a+b-1$ is odd and greater than $1$, as $a,b > 1$. It is well known that in this case $B_{a+b-1}=0$, from which we deduce that   $\lambda^1_{a,b}+\lambda^1_{b,a}=0$.
\end{proof}

\begin{thm} 
Let $Z(Q,\N_{>1})$  be any family of q-analogues of multiple zeta values as in \eqref{eq:zqs}, where 
each $Q_s(t) \in Q$ is a polynomial with degree at most $s-1$, then 
 \[ Z(Q,\N_{>1}) =   \MD^\sharp \,. \]
  and therefore all such  families of q-analogues of multiple zeta values  are $\Q$-sub algebras of $\MD$. 
  In particular  $\qMZV =  \MD^\sharp$.
\end{thm}
\begin{prf} To proof the first equality it is sufficient to show that for each $s>1$ there are numbers $\lambda_j \in \Q$ with $2 \leq j \leq s$ such that 
\[ \frac{Q_s(t)}{(1-t)^s} = \sum_{j=2}^{s} \lambda_j \frac{Q^E_j(t)}{(1-t)^j} \,. \]
The space of polynomials with at most degree $s-1$ and no constant term has dimension $s-1$. For  $2 \leq j \leq s$ the polynomials $(1-t)^{s-j} Q'_j(t)$ are all linear independent since $Q'(1)\neq 1$ and therefore such $\lambda_j$ exist. The second statement follows directly  from the definition of $\qMZV$. 
\end{prf}

The following proposition allows one to write an arbitrary element in $Z(Q,\N_{>1})$ as an linear combination of $[s_1,\dots,s_l] \in \MD^\sharp $.
\begin{prop} Assume $k\geq 2$.  For
$1 \leq i,j \leq k-1$ define the numbers $b^{k}_{i,j} \in \Q$ by
\[\sum_{j=1}^{k-1} \frac{b^k_{i,j}}{j!} t^j :=\binom{t+k-1-i}{k-1}
% \frac{(t-i+1)_{(k-1)} }{(k-1)!} = \binom{t+k-1-i}{k-1}
\,.\]
%where $(t)_{(n)} = x (x+1) \dots (x+n-1)$ is the Pochhammer symbol. 
With this it is for $1 \leq i \leq k-1$ and $Q^E_{j}(t) = \frac{1}{(j-1)!} t P_j(t)$ 
\[ t^i = \sum_{j=2}^k b^k_{i,j-1} (1-t)^{k-j} Q^E_{j}(t) \,.\] 
\end{prop}
\begin{prf}
We want to show that 
\[ \frac{t^i}{(1-t)^k} = \sum_{j=1}^{k-1}  \frac{b^k_{i,j}}{j!} \frac{t P_{j}(t)}{(1-t)^{j+1}} \] 
By the definition of the Eulerian Polynomials it is
\begin{align*}
\sum_{j=1}^{k-1} \frac{ b^k_{i,j}}{j!} \frac{t P_{j}(t)}{(1-t)^{j+1}} &= \sum_{j=1}^{k-1}  \frac{b^k_{i,j}}{j!} \sum_{d>0} d^j t^d \\
&=  \sum_{d>0} \left( \sum_{j=1}^{k-1} \frac{ b^k_{i,j}}{j} d^j \right) t^d \\
&= \sum_{d>0} \binom{d-i+k-1}{k-1} t^d
\end{align*}
The claim now follows directly from the easy to prove formula
\[ \frac{1}{(1-t)^k} = \sum_{n\geq 0} \binom{n + k-1}{k-1} t^n \,. \]
\end{prf}
We give some examples how to write elements in $\qMZV$ as linear combinations of elements in $\MD$. From the proposition we deduce for the length one case for all $k>0$ 
\begin{align*}
\OZ(2k) = \sum_{j=2}^{2k} b_{k,j-1}^{2k} [j] \,\quad \text{ and } \quad \OZ(2k+1) = \sum_{j=2}^{2k+1} \left( b_{k,j-1}^{2k+1} + b_{k+1,j-1}^{2k+1} \right) [j] \,.
\end{align*}
Clearly this also suffices to give linear combinations in higher length. 
\begin{ex} We give some examples
\begin{align*}
\OZ(2) &= [2] \,, \qquad \OZ(3) = 2 [3] \,,\\
\OZ(4) &= [4] - \frac{1}{6} [2] \,, \qquad \OZ(5) = 2 [5] - \frac{1}{6} [3] \,, \\
\OZ(6) &= [6] - \frac{1}{4} [4] + \frac{1}{30} [2] \,, \qquad \OZ(7) = 2[7] - \frac{1}{3} [5] + \frac{1}{45} [3] \,,\\
\OZ(2,2) &= [2,2] \,, \qquad \OZ(2,4) = [2,4] - \frac{1}{6} [2,2]  \,. \\
\end{align*}
\end{ex}

The $q$-expansion of  modular forms are well known to give rise to $q$-analogues of Riemann zeta values. 
Let us denote by $M_\Q = \Q[G_4,G_6]$ and $\widetilde{M}_Q = \Q[G_2,G_4,G_6]$  the ring of modular and quasi-modular forms, where the Eisenstein series $G_2$, $G_4$ and $G_6$ are given by
\[ G_2 = -\frac{1}{24} + [2] \,,\quad G_4 = \frac{1}{1440} + [4] \,, \quad G_6 = -\frac{1}{60480} + [6] \,.\]

We clearly have the following inclusions of $\Q$-algebras 
\[ M_\Q \subset \widetilde{M}_\Q  \subset \qMZV \subset \MD \,. \]
where the second inclusion follows from 
\begin{align*}
G_2 &= -\frac{1}{24} + Z(2) \,,\\
G_4 &=  \frac{1}{1440} + Z(2) + \frac{1}{6} Z(4) \, ,\\
G_6 &=  -\frac{1}{60480} + Z(6) + \frac{1}{4} Z(4) + \frac{1}{120} Z(2)  \,.
\end{align*}
In the theory of modular forms the operator $\mathrm{d}:=q \frac{d}{d q}$ plays an important role   and it is a well known fact that $\widetilde{M}_Q$ is closed under $\dif$. 

\begin{prop}
The subalgebra of the quasi-modular forms $\widetilde{M}_\Q\subset \MD$ is graded by the weight and filtered by the length and its Hilbert series satisfies
\begin{align}\label{HgrwlM}
  \sum_{k,l}  \dim_\Q \grwl_{k,l} \widetilde{M}_\Q   \, x^k t^l =  
\Big(   1 + \frac{x^4}{1-x^2} \, t + \frac{x^{12}}{(1-x^4)(1-x^6)} \, t^2 \Big)  
\Big( \frac{1}{1-x^2\, t} \Big).
\end{align}
\end{prop}
\begin{proof} Since $\widetilde{M}_Q = \Q[G_2,G_4,G_6]$, 
the formula 
\eqref{HgrwlM} follows from the fact that $M_\Q$ is spanned as a vector space by products of Eisenstein series of the form $G_a G_b$. 
\end{proof}

In \cite{bk} the authors showed the following 
\begin{thm} \label{thm:derivative} The operator $\dif$ is a derivation on $\MD$ that is compatible with the filtrations on $\MD$ given by the weight and the length. 
%More precisely it maps $\filwle_{k,l}(\MD)$ to $\filwle_{k+2,l+1}(\MD)$. 
\end{thm}

In \cite{ao} the following conjecture is stated by Okounkov

\begin{conj}\label{Ok_derconj}
The operator $\dif$ is a derivation on $\qMZV$.
\end{conj}

For the derivative of a length one generating series of multiple divisor sums we have several identies. These
will be used to make the following result which gives some evidence for the conjecture above. 

\begin{prop}
It is $\dif \OZ(k) \in \qMZV$ for all $k\geq 2$. 
\end{prop}
\begin{prf}
In \cite{bk} Theorem 3.5 the authors prove the following representation of the derivative $\dif[k-2]$ 
\begin{align*}
\binom{k-2}{s_1-1} &\frac{\dif[k-2]}{k-2} = [s_1]\cdot[s_2] - [s_1,s_2]-[s_2,s_1]  \\
&+\binom{k-2}{s_1-1} [k-1] - \sum_{\substack{a+b=k \\ a > s_1}} \left( \binom{a-1}{s_1-1}+\binom{a-1}{s_2-1} - \delta_{a,s_2} \right) [a,b]  \,. 
\end{align*}
where $s_1,s_2 > 0$ can be choosen arbitrary such that $k=s_1 + s_2$. First divide both sides by $\binom{k-2}{s_1-1} (k-2)^{-1}$. Whenever $k\geq 4$ all elements on the right of the resulting equation belong to $\MD^\sharp$ except for the term with $[k-1,1]$. By direct calculation one obtains that for $s_1=1$ and $s_2=k-1$ the coefficient of $[k-1,1]$ is $-(k-2)$ and for $s_2=2$ and $s_2=k-2$ it is $-2(k-2)$ and therefore $\dif[k-2]$ can be expressed as an element in $\MD^\sharp$.  
\end{prf}

 Since $\dif$ is a derivation it  satisfies the Leibniz rule. Therefore 
the above proposition allows us to derive further identites, e.g.
\[
\dif \QZ (k,...,k), \dif \left( \QZ (k_1,k_2)+ \QZ (k_2,k_1) \right) \in \qMZV.
\]

\begin{ex} Some examples of representations of $\dif \OZ(s)$ in $\qMZV$.
\begin{align*}
\dif Z(2) &= 3 Z(4) + Z(2) - Z(2,2) \,,\\
\dif Z(3) &= 5 Z(5) + Z(3) - 4 Z(3,2) - 6 Z(2,3) \,, \\
\dif Z(4) &= 10 Z(6) + 2 Z(4) + 4 Z(4,2) - 8 Z(2,4) - 6 Z(3,3) \,, \\
\dif Z(2,2) &= -6 Z(6) - 12 Z(2,2,2) - 15 Z(4,2) + 3  Z(2,4) + 9 Z(3,3) \,, \\
\dif Z(3, 3) &= 4Z(8)-12Z(2, 3, 3)-10Z(3, 2, 3)-8Z(3, 3, 2)\\
&+Z(3, 5)-Z(5, 3)+8Z(6, 2)+3Z(3, 3) \,, \\
\dif Z(2, 2, 2) &= -24Z(2, 2, 2, 2)+9Z(2, 3, 3)+9Z(3, 2, 3)+6Z(3, 3, 2)\\
&-15Z(4, 2, 2)-15Z(2, 4, 2)+3Z(2, 2, 4)-6Z(2, 6)+6Z(5, 3)-6Z(6, 2) \,.
\end{align*}
\end{ex}
At the end we give some conjectured representations of  $\dif \OZ(s)$ in $\qMZV$ coming from numerical experiments and which where checked for the first $200$ coefficients but which should be also provable by using the results in \cite{bk}:   
\begin{align*}
\dif Z(2,3) &= 2Z(7)-16Z(2, 2, 3)-4Z(2, 3, 2)-8Z(3, 2, 2)\\
&-15Z(4, 3)-4Z(3, 4)+4Z(5, 2)+5Z(2, 5)+Z(3, 2)-Z(2, 3) \,.
\end{align*}

\section{A refined conjecture and numerical evidences}
In his article Okounkov states the following conjecture:

\begin{conj}(Okounkov)\label{Okdimconj} \begin{enumerate}[i)]
\item   The algebra $\qMZV$ is spanned by $Z(s)$ with $2 \le s_i \le 5$.  
 
\item  The Hilbert series of the graded algebra $\grw \qMZV$ equals
 \begin{align}\label{Okdim}
\!\!\!\!\!\!\!\!\!  \sum_k  \dim_\Q \grw_k \!\!\qMZV  \, x^k = \frac{1}{ 1- x -x^2 -x^3 -x^4 -x^5 + x^8 +x^9+  x^{10}+  x^{11} +  x^{12} }.
  \end{align}
  
\end{enumerate}  
\end{conj}

We conjecture in addition 
\begin{conj} \label{BKdim}
\begin{enumerate}[i)]
\item The algebra $\qMZV$ is isomorphic to a free graded polynomial algebra. 
  \item The algebra $\qMZV$ is isomorphic to the  tensor product of the algebra 
of quasi-modular forms $\widetilde{M}_\Q$ with a graded algebra $\mathcal{A}$ that has the Hilbert series
 \begin{align}\label{HgrwA}
  \sum_k  \dim_\Q \grw_k \!\!\mathcal{A}  \, x^k =  \frac{1}{ 1-\displaystyle{{\frac {{x}^{3}}{ \left( 1-{x}^{2} \right) ^{2}}}+2\,{\frac {{x}^{
12}}{ \left( 1-{x}^{4} \right)  \left(1 -{x}^{6} \right)  \left( 1-{x
}^{2} \right) }}}
    }.
\end{align}
\item    
The algebra  $\mathcal{A}$ is graded by the weight and by the length and it satisfies
\begin{align}\label{HgrwlA}
  \sum_{k,l}  \dim_\Q \grwl_{k,l} \!\!\mathcal{A}   \, x^k t^l =  
\frac{1}{\displaystyle{ 1-   
  {\frac {{x}^{3}t}{ \left( 1-{x}^{2} \right)  \left(1 -{x}^{2}t
 \right) }}+{\frac {{x}^{12} \left( {t}^{3}+{t}^{2} \right) }{ \left( 
1- {x}^{4} \right)  \left( 1-{x}^{6}\right)  \left(1 -{x}^{2}t
 \right) }}}}
.
\end{align}
The algebra $\qMZV$ is therefore graded by the weight and filtered by the length and its Hilbert series equals the product of \eqref{HgrwlM} and \eqref{HgrwlA}.

\end{enumerate}
\end{conj}

It is easy to see that \eqref{HgrwlA} implies \eqref{HgrwA} by setting $t=1$. 
Moreover \eqref{HgrwA} and the well-known dimension formula for the quasi-modular forms imply \eqref{Okdim}.  

\begin{rem}
One should view Okounkovs conjectures \ref{Okdimconj} as analogues of the Hoffmann conjecture, now Brown's theorem, and the Zagier conjecture for multiple zeta values. In this context our conjectures\ref{BKdim}
would then correspond to the Broadhurst-Kreimer conjecture. In \cite{bk3} we study similar conjectures for the algebra $\MD$ and the algebra of bi-brackets, which are a natural generalization of the brackets  \cite{Babibra}.
\end{rem}

With the Computer algebra package \cite{pari} we calculated  in table \ref{tab:filok-bounds} lower bounds for 
  $\filwle_{k,l} \qMZV$ and according to the above conjectures these  numbers should be the exact dimensions. The same holds for similar tables with different maximal values for the weight $k$ and the length $l$, such that $\dim_\Q  \filwle_{k,l} \qMZV < 5000$.

%\hspace{-2cm}
\begin{table}[H]
\centering
\resizebox{1.\textwidth}{!}{
\begin {tabular} {c|ccccccccccc} 
 k $\backslash$ l&1&2&3&4&5&6&7&8&9&10&11
\\ \hline
 1&1&0&0&0&0&0&0&0&0&0&0\\  2&2&2
&0&0&0&0&0&0&0&0&0\\  3&3&3&3&0&0&0&0&0&0&0&0
\\  4&4&5&5&5&0&0&0&0&0&0&0\\  5&5&8
&8&8&8&0&0&0&0&0&0\\  6&6&11&12&12&12&12&0&0&0&0&0
\\  7&7&15&19&19&19&19&19&0&0&0&0
\\  8&8&19&27&28&28&28&28&28&0&0&0
\\  9&9&24&38&43&43&43&43&43&43&0&0
\\  10&10&29&51&63&64&64&64&64&64&64&0
\\  11&11&35&67&90&96&96&96&96&96&96&96
\\  12&12&41&85&125&142&143&143&143&143&143&143
\\  13&13&48&107&170&206&213&213&213&213&213&213
\\  14&14&55&132&225&293&316&317&317&317&317&317
\\  15&15&63&160&292&408&462&470&470&470&470&470
\\  16&16&71&192&374&557&667&697&698&698&698&698
\\  17&17&80&228&470&745&947&1025&1034&1034&1034&1034
\\  18&18&89&268&584&983&1323&1494&1532&1533&1533&
1533\\  19&19&99&312&717&1275&1815&2151&2260&2270&
2270&2270\\  20&20&109&361&871&1632&2455&3057&3314&
3361&3362&3362\\  21&21&120&414&1047&2064&3271&4280&
4818&4966&4977&4977.
\end{tabular}}
\caption{Lower bounds for $\dim_\Q  \filwle_{k,l}\qMZV$ \label{tab:filok-bounds}}
\end{table}

Finally we remark, that in accordance with Okounkov's conjecture \ref{Ok_derconj}
adjoining all derivatives does not increase 
the lower bounds for $ \dim \filw_k \qMZV$  up to $k=19$.

{\small
{\it Addresses:}\\\texttt{henrik.bachmann@math.nagoya-u.ac.jp}\\
\noindent {\sc Graduate School of Mathematics\\
Nagoya University\\
Chikusa-ku, Nagoya, 464-8602\\
Japan}\\
\texttt{kuehn@math.uni-hamburg.de}\\
\noindent {\sc Fachbereich Mathematik (AZ)\\ Universit\"at Hamburg\\ Bundesstrasse 55\\ 20146 Hamburg\\
Germany}}

\end{document}